\documentclass[11pt, a4paper]{amsart}
\usepackage{amssymb,latexsym,amsmath,graphicx} 
\usepackage{enumerate,enumitem,color} 
\usepackage{hyperref} 
\usepackage{mathrsfs} 
\usepackage{calligra}
\usepackage[T1]{fontenc}

\usepackage{tikz,float}
\usetikzlibrary{shapes,positioning}

\input{xypic}
\xyoption{all}

\usepackage{fullpage}
\usepackage{mathtools}
\usepackage{amssymb,mathrsfs,tikz-cd,amsthm,hyperref,stmaryrd,enumitem,mathdots,xcolor}

\newcommand{\fcy}{\mathbb}

\newcommand{\mf}{\mathfrak}

\newcommand{\cc}{\mathbb{C}}
\newcommand{\qq}{\mathbb{Q}}
\newcommand{\zz}{\mathbb{Z}}

\DeclareMathOperator{\gal}{Gal}

\DeclareMathOperator{\aut}{Aut}

\DeclareMathOperator{\Hm}{Hom}

\DeclareMathOperator{\sgn}{sgn}

\newtheorem{prop}{Proposition}[section]

\newtheorem{lem}{Lemma}[section]

\newtheorem{prob}{Problem}[section]

\def\zz{\mathbb{Z}}
\def\Q{\mathbb{Q}}
\def\qq{\mathbb{Q}}

\numberwithin{equation}{section}   

\newtheorem{theorem}[equation]{Theorem}

\title[\bf Totally imaginary number fields]
{On counting totally imaginary number fields}

\author{\bf A. Raghuram, \ \ Qiyao Yu}

\date{\today}      
\subjclass[2010]{11R45; 11R16, 11R21, 11R32}

\address{Dept.\,of Mathematics, Fordham University at Lincoln Center, New York, NY 10023, USA.} 

\email{araghuram@fordham.edu}

\address{Dept.\,of Mathematics, Columbia University, 2990 Broadway, New York, NY 10027, USA.}

\email{qy2266@columbia.edu}

\begin{document}

 \begin{abstract}
A number field is said to be a CM-number field if it is a totally imaginary quadratic extension of a totally real number field. 
We define a totally imaginary number field to be of \textbf{CM}-type if it contains a CM-subfield, and of \textbf{TR}-type if it does not contain a CM-subfield. For quartic totally imaginary number fields when ordered by discriminant, we show that about 69.95\% are of 
\textbf{TR}-type and about 33.05 \% are of \textbf{CM}-type. For a sextic 
totally imaginary number field we classify its type in terms of its Galois group and possibly some additional information about the location of complex conjugation in the Galois group. 
\end{abstract} 

\maketitle

\section{Introduction}

The past few decades have seen a huge amount of work on counting number fields; see, for example, \cite{ali-s-v-w}, \cite{bhargava}, \cite{cohen}, 
\cite{dummit}, \cite{ellenberg-venkatesh}; these represent only a tiny sampling of a huge and growing body of related literature which the interested reader can search on MathSciNet. 
In this article, we introduce a new aspect to the general problem of counting number fields by restricting our attention to totally imaginary number fields and then searching for those which contain a CM subfield.

\medskip

A number field will mean a finite extension of $\Q$. A number field 
is said to be a CM-number field if it is a totally imaginary quadratic extension of a totally real number field. A totally imaginary number field is said to be of \textbf{CM}-type if it contains a CM-subfield, and is said to be of \textbf{TR}-type if it does not contain a CM-subfield. One can organize this definition in other ways. 
Let $F$ be a totally imaginary number field; since the compositum of two totally real fields is again totally real, $F$ contains a maximal totally real subfield, say, 
$F_0$. There is at most one totally imaginary quadratic extension $F_1$ of $F_0$ contained inside $F$; if such a CM field $F_1$ exists, then $F$ is of \textbf{CM}-type and $F_1$ is in fact the maximal CM subfield of $F$; otherwise, we say $F$ is of \textbf{TR}-type, and we put $F_1 = F_0.$ This dichotomy of totally imaginary number fields played an important role in \cite{raghuram}. 

\medskip

Let $n \geq 2$ be an even positive integer. To count totally imaginary number fields $F$ of prescribed degree $n=[F:\qq]$ and bounded discriminant $\delta_F\leq X$, define 
$N_n^{\textbf{TI}}(X)$ to be the number of isomorphism classes of totally imaginary number fields $F$ of degree $[F:\qq]=n$, with absolute discriminant $\delta_F\leq X$. 
To count those totally imaginary fields of \textbf{CM}-type, define $N_n^{\textbf{CM}}(X)$ to be the 
number of isomorphism classes of totally imaginary number fields $F$ of $\textbf{CM}$-type of degree $n$ and $\delta_F\leq X$. We would like to know the asymptotics of the ratio
\[
R_n^{\textbf{CM}}(X):=\frac{N_{n}^{\textbf{CM}}(X)}{N_{n}^{\textbf{TI}}(X)}, \quad \mbox{as $X\to\infty$}. 
\]
 The question is trivial for $n = 2$ since an imaginary quadratic field is a CM-field; $R_2^{\textbf{CM}}(X) = 1.$ Henceforth, assume that $n \geq 4.$

\medskip
For quartic extensions ($n=4$), in Theorem~\ref{thm:quartic} we show that $\lim_{X \to \infty} R_4^{\textbf{CM}}(X) \approx 0.33052$. The proof uses the results of 
Bhargava \cite{bhargava} and Cohen--Diaz Y Diaz--Olivier \cite{cohen}. 

\medskip
For sextic extensions ($n=6$), in Theorem~\ref{thm:sextic}, we classify the 
type of a totally imaginary extension in terms of its Galois group and possibly some additional data. If 
$F$ is a sextic totally imaginary extension, and $F_s$ is its Galois closure, 
then the Galois group of $F_s/\Q$ is a transitive subgroup of $S_6$ containing 
a $(2,2,2)$ cycle. There are $11$ possible such subgroups of $S_6$ (up to conjugacy), and for each one of them we can classify the type of $F$. In particular, we show  that if 
$G$ is $S_4,$ $S_5,$ or $S_6$ then $F$ is of \textbf{TR}-type; 
if $G$ is $C_6,$ $S_3$, $D_6,$ $S_3 \times C_3,$ $A_4 \times C_2,$ $S_3 \times S_3,$ or  $C_3^2 \rtimes D_4,$ 
then $F$ is of \textbf{CM}-type, and furthermore in each case one can decide if $F$ itself is a CM field or whether it only contains a proper CM-subfield in terms of the Galois group and its internal structure. The case of $D_6$ offers some piquant features. Finally, the case $G = S_4 \times C_2$ is more complicated because there are totally imaginary fields of either type with this Galois group; identifying an order $2$ element of $G$ that represents complex conjugation serves to distinguish between the two types. Finally, within the framework of currently 
available results on counting number fields we are unable to predict a precise value for  $\lim_{X \to \infty} R_6^{\textbf{CM}}(X).$ Even assuming Malle's conjecture (see \cite[Conj.\,1.1]{malle})
on counting number fields it seems difficult to predict a value because of examples like the $S_4 \times C_2$ case, and also because we do not apparently know--even conjecturally--the constants of proportionality denoted $c(k,G)$ for $k = \qq$ in \cite{malle}. 

\bigskip

\noindent
\textit{Acknowledgements:} AR is grateful to the warm hospitality of the Math Department of Columbia University where much of this work was carried out.

\bigskip
\section{Preliminaries}

\subsection{Notations}
The cylic group of order $n$ is denoted $C_n.$ The Klein $4$-group $C_2 \times C_2$ is denoted $V_4.$ 
The dihedral group of order $2n$ is $D_n = \{1,r,r^2,\dots, r^{n-1}\} \rtimes \{1,s\}$, with 
$r^n = s^2 = 1$ and $srs = r^{-1}.$
The permutation group on $n$ letters, usually taken to be $\{1,2,\dots,n\}$, 
is denoted $S_n$; and $A_n$ denotes the alternating group--the subgroup of $S_n$ of all even permutation.

\subsection{Examples}
A number field will always mean a finite extension of $\Q$, and by its Galois group one means the Galois group of its normal closure. 
There are several families of examples of totally imaginary number fields of the two types as will become clear throughout the paper. We found it especially helpful to search the LMFDB (the $L$-functions and modular forms database \cite{lmfdb}) for number fields with prescribed constraints, such as degree, signature, Galois group, and whether or not it is a CM field. 
For example, we can set the degree to be $4$, the signature as 
$[0,2]$, the Galois group as $S_4$ or also by its label 4T5 (the fifth group amongst the transitive subgroups of $S_4$) and then upon hitting search we get a list ordered by discriminant of quartic totally imaginary number fields with Galois group $S_4$--the first item in this list is the number field labelled as 4.0.229.1; the parameters in the label stands for degree, number of real embeddings, absolute discriminant, and the last one is item number for all 
the number fields with the given first three parameters. As another example, if we search for number fields with degree $6$, 
signature $[0,3]$, and Galois group $A_6$ (label 6T15), then we get `no matches'. Indeed, if $F$ is a number field of signature $[r_1,r_2]$, i.e,
$r_1$ real embeddings and $r_2$ pairs of complex embeddings, and suppose 
$r_2$ is odd then $A_n$ cannot be the Galois group of $F$, where $n = r_1 + 2r_2$; see Lem.\,\ref{lem:A_n_not_Galois}.

\bigskip
\section{Quartic extensions}

In this section, we study the case of quartic number fields; $n = 4.$

A number field is said to be primitive if it admits no intermediate subfields. Clearly, a totally imaginary primitive number field is of {\bf TR}-type. 
The number field labelled 4.0.229.1 in LMFDB mentioned above is an example of a totally imaginary quartic primitive field and hence of \textbf{TR}-type. 
Given a totally imaginary quartic extension $F/\qq$, it is either primitive or contains a quadratic subfield. If $F$ contains a real quadratic extension, then $F$ is a CM-field; and if $F$ contains an imaginary quadratic extension--which is necessarily a CM-field then $F$ is of \textbf{CM}-type. Hence, if $F$ is not primitive then it is of \textbf{CM}-type; whence: 
\[F\text{ is of \textbf{TR}-type} \ 
\ \iff \
F\text{ is primitive}.\]

Let $F$ be a quartic totally imaginary number field and $F_s$ its Galois closure. Then the Galois group $G = \gal(F_s/\qq)$ is a transitive subgroup of $S_4$ via its action on 
the set of $4$ roots of an irreducible quartic polynomial of a primitive element generating $F$. Up to isomorphism, $G$ is one of 
$C_4,$ $V_4,$ $D_4,$ $A_4,$ or $S_4$. Let us consider each case. 
If $G$ is $C_4$ or $V_4$ then it is not primitive and hence of \textbf{CM}-type. If $G$ is $D_4$ then $H = \gal(F_s/F)$ is of order $2$, and any order $2$ subgroup of $D_4$ is contained in some subgroup, say $K$, of order $4$ of $D_4$; then $F_s^K$ is a quadratic subfield of $F$; whence $F$ is of \textbf{CM}-type. 
If $G$ is $A_4$ (resp., $S_4$) then $H = \gal(F_s/F)$ corresponds to a subgroup of order 3 (resp., 6); such a subgroup is maximal; hence by Galois theory $F$ is primitive. To summarize, we have shown 
\[
\text{$F$ is primitive $\iff$ $\gal(F_s/\qq)$ is $S_4$ or $A_4$.}
\]
The argument above does not rely on the fact that $F$ is totally imaginary and but only employs Galois theory and elementary group theory. 

Thus, to count quartic number fields of \textbf{TR}-type (which is equivalent to counting \textbf{CM}-type), we need to know the asymptotics for totally imaginary quartic fields with Galois group $S_4$ and $A_4$. Towards this, let us recall some results from the literature on such asymptotics; in the results quoted below, number fields are ordered by discriminant. 
\begin{enumerate}
\item Bhargava \cite{bhargava} showed that approximately $17.111\%$ of quartic fields have Galois group $D_4$, and the remaining $82.889\%$ of quartic fields have Galois group $S_4$; in particular, the contribution from the $C_4$, $V_4,$ and $A_4$ quartic fields are negligible in comparison to the $S_4$ and $D_4$ quartic fields. 

\item Bhargava also proved that asymptotically $30\%$ of the $S_4$-quartic fields are totally imaginary; see Thm.\,1 of {\it loc.\,cit.}

\item Cohen--Diaz Y Diaz--Olivier \cite[Prop.\,6.2]{cohen} proved that asymptotically $71.747\%$ of the $D_4$-quartic fields are totally imaginary. 
\end{enumerate}
From these results one can deduce the first main theorem of this article:

\begin{theorem}
\label{thm:quartic}
 When ordered by absolute discriminant, asymptotically, about $66.948\%$ of the totally imaginary quartic fields are of \textbf{TR}-type, and the remaining $33.052\%$ 
 of the totally imaginary quartic fields are of \textbf{CM}-type, i.e., 
 $
 \lim_{X \to \infty} R_4^{\textbf{CM}}(X) \approx 0.33052.$
\end{theorem}

\begin{proof}
Using the results quoted above, the proof is a simple exercise in Bayes' theorem. Within the sample space of all quartic fields ordered by discriminant (taken to be 
bounded by $X$ and then letting $X$ go to infinity to get the required asymptotics), let $D$ stand for the subset of $D_4$ extensions, $S$ stand for the subset of $S_4$ extensions, 
and $T$ stand for the subset of totally imaginary extensions. The results quoted can be rephrased as:
\begin{enumerate}
\item $P(D) = 0.17111$ and $P(S) = 0.82889$; 
\item $P(T|S) = 0.3$; 
\item $P(T | D) = 0.71747$.
\end{enumerate}
We compute then that 
\[P(S|T) = \frac{P(T|S) \cdot P(S)}{P(T|S) \cdot P(S) + P(T|D) \cdot P(D)} = 
\frac{0.30 \times 0.82889}{0.30 \times 0.82889 + 0.71747 \times 0.17111} = 0.66948,  
\]
i.e.,  of all totally imaginary quartic fields, about 66.948\% have Galois group $S_4$ and hence of {\bf TR}-type and the remaining 33.052\% have Galois 
group $D_4$ and hence of {\bf CM}-type. 
\end{proof}

\section{Sextic totally imaginary extensions}

Before we can state our main theorem for sextic totally imaginary extensions, we need to collect some preliminaries on $6$-transitive groups. 

\subsection{$6$-transitive groups}
Let $F$ be a sextic totally imaginary number field and $F_s$ its Galois closure. Then the Galois group $G = \gal(F_s/\qq)$ is a transitive subgroup of $S_6$ via its action on 
the set of $6$ roots of an irreducible quartic polynomial of a primitive element generating $F$. Up to isomorphism, there are $16$ possible $6$-transitive subgroups of $S_6$. However, 
the six roots being pairwise conjugate, there is an element of $G$ of $(2,2,2)$-cycle decomposition as an element of $S_6$, in particular $G$ is not a subgroup of $A_6$; this 
cuts down the set of possible $6$-transitive groups for $G$ to be one of the following eleven:  

\begin{enumerate}
\item (6T1) $C_6$
\item (6T2) $S_3$
\item (6T3) $D_6$
\item (6T5) $S_3\times C_3$ 
\item (6T6) $A_4\times C_2$ 
\item(6T8)  $S_4$ 
\item (6T9) $S_3\times S_3$ 
\item (6T11) $S_4\times C_2$ 
\item (6T13) $C_3^2\rtimes D_4$
\item (6T14) $S_5\cong PGL_2(\fcy{F}_5)$
\item (6T16) $S_6$ 
\end{enumerate}
The above list can be deduced from the corresponding table in 
Dummit \cite[App.\,A]{dummit}. The label (6Tn) refers to the label of the group as in the database of $6$-transitive groups in Magma or GAP. 

\subsection{A digression on possible Galois groups}
That $A_6$ is not the Galois group of a sextic totally imaginary number field is a special case of the following more general lemma: 

\begin{lem}
\label{lem:A_n_not_Galois}
    Let $F$ be a number field of signature $(r_1,r_2)$ and degree $n=[F:\qq]=r_1+2r_2$. Let $G=\gal(F_s/\qq)$ be the Galois group of a Galois closure $F_s$ of $F$.    
    If $r_2$ is odd, then $G\not\cong A_n$.
\end{lem}

\begin{proof}
The number field $F$ is defined by an irredcuble polynomial $f(x)\in \qq[x]$.  The polynomial discriminant, $\mathrm{disc}(f)$, is well-defined up to non-zero squares. The  discriminant root field of the extension $\qq(\sqrt{\mathrm{disc}(f)})$ is well-defined. The Galois group $G$ is a subgroup of $S_n$, well-defined up to conjugation. The discriminant root field is 
the fixed field of $\gal(F_s/\qq)\cap A_n$. Note that $\mathrm{sign}(\delta_F)=(-1)^{r_2}$ and $\mathrm{sign}(\delta_F)=\mathrm{sign}(\mathrm{disc}(f))$. If $r_2$ is odd, then $\mathrm{sign}(\mathrm{disc}(f))=-1$. Thus, $\qq\left(\sqrt{\mathrm{disc}(f)}\right)$ is an imaginary quadratic extension of $\qq$. Suppose $\gal(F_s/\qq)=A_n$, then by Galois theory 
$\gal\left(F_s/\qq\left(\sqrt{\mathrm{disc}(f)}\right)\right) \cong \gal(F_s/\qq)\cap A_n = \gal(F_s/\qq),$ which implies $\qq\left(\sqrt{\mathrm{disc}(f)}\right)=\qq$, giving us a contradiction. 
\end{proof}
In particular, if $F$ is a totally imaginary number field of degree $n$, and $n/2$ is odd, then $\gal(F_s/\qq)\not\cong A_n$. This applies to our case of $n=6$. This ends the digression. 

\subsection{The main result on sextic totally imaginary number fields}

\begin{theorem}
\label{thm:sextic}
Let $F$ be a sextic totally imaginary number field, $F_s$ its Galois closure, and $G = \gal(F_s/\qq)$ its Galois group. 
\begin{enumerate}
\smallskip
\item If $G$ is $S_4$, $S_5$, or $S_6$ then $F$ is of {\bf TR}-type. 
\smallskip
\item If $G$ is $C_6$, $S_3$, $D_6$, $S_3 \times C_3$, $A_4 \times C_2$,  $S_3 \times S_3$, or  $C_3^2 \rtimes D_4$ then $F$ is of {\bf CM}-type. 
\smallskip
\item If $G = S_4 \times C_2$ then $F$ can be of either type; one can distinguish between the two types by identifying an order $2$ element of $G$ that represents complex conjugation. 
\end{enumerate} 
\end{theorem}

\subsection{Proof of Theorem~\ref{thm:sextic}}

\subsubsection{Preliminary remarks}
We still have the basic observation 
\[F\text{ is primitive }\implies F\text{ is of \textbf{TR}-type}.\]
Unlike the degree 4 case, examples of $F$ of \textbf{TR}-type but not primitive do exist\footnote{LMFDB number field 6.0.29095.1}.
If $F$ is not primitive, then it can only contain quadratic or cubic subextensions. Quadratic subextensions of $F$ can in principle be either real or imaginary; but if a quadratic subextension is real, then $F$ is a cubic totally imaginary extension of a real quadratic field, which by ramification theory is not possible. Thus, the maximal totally real subfield $F_0$ can either be $\qq$ or a totally real cubic subextension. If $F$ is not primitive and does not contain a cubic subextension, then the maximal totally real subfield $F_0$ of $F$ is $\qq$, and $F$ is of \textbf{CM}-type. If $F$ is not primitive and does contain a cubic subextension $F'$, then $F'$ can be totally real or of mixed signature; if $F'$ is totally real, $F$ is of \textbf{CM}-type. Otherwise, 
    i.e., if $F$ not primitive and every cubic subfield has mixed signature, then 
    $F$ is of \textbf{TR}-type if $F$ does not contain a quadratic subextension. 
    In conclusion, $F$ is of \textbf{TR}-type if and only if $F$ contains neither a quadratic extension nor a totally real cubic extension.
    

\subsubsection{$G = C_6$}
If $\gal(F_s/\qq)\cong C_6$, then $F_s=F$ and the subfields of $F$ are Galois. Both quadratic and cubic subfields occur as they correspond to the index 2 and index 3 subgroups of $C_6$ respectively. By our analysis above, quadratic subfields can only be imaginary. Moreover, a Galois cubic field has to be totally real. Therefore, $F$ is a CM-field.

\subsubsection{$G = S_3$}
If $\gal(F_s/\qq)\cong S_3$, then $F_s=F$. The imaginary quadratic subfield of $F$ corresponds to the subgroup $A_3$ of $S_3$ and any cubic subfield of $F$ corresponds to a copy of $S_2$ in $S_3$. Suppose $K$ is a totally real cubic subfield of $F$, then $K/\qq$ is Galois, because, for $\sigma\in\gal(F/\qq)$, we know that $\sigma(K)$ is also totally real, and 
if $\sigma(K) \neq K$ then $F = K\sigma(K)$ is totally real which is not the case, hence $\sigma(K) = K$, whence $K/\Q$ is Galois; this implies that $\gal(F_s/K)$ is an order $2$ normal subgroup
of $S_3$ which is not possible. Thus, only mixed signature cubic subfields can appear; hence $F$ is of {\bf CM}-type but not a CM field.

\subsubsection{$G = S_5$}
If $\gal(F_s/\qq)\cong S_5$, we will show that $F$ is primitive, hence of \textbf{TR}-type. 
Specifically, $F$ corresponds to a subgroup in $S_5$ of order 20. Suppose $F$ is not primitive, then there exists a subgroup of $\gal(F_s/\qq)$ of order 40 or 60 containing $\gal(F_s/F)$ as a maximal subgroup. The maximal subgroups of $S_5$ are of order 12, 20, 24, or 60\footnote{https://groupprops.subwiki.org/wiki/Subgroup\_structure\_of\_symmetric\_group:S5}.
In particular, there is no subgroup of order 40. The only subgroup of $S_5$ of order 60 is $A_5$. But then $\gal(F_s/F)$ would be a subgroup of order 20 inside $A_5 = \gal(F_s/F_s^{A_5})$; but $A_5$ does not have a subgroup of order 20. Thus, $F$ is primitive.

\subsubsection{$G = S_6$}
 If $\gal(F_s/\qq)$ is $S_6$, then as in the $S_5$ case, we will see that $F$ is primitive, hence of \textbf{TR}-type.
    If $\gal(F_s/\qq)\cong S_6$, then $\gal(F_s/F)$ is a subgroup of $S_6$ of order 120. Suppose $F$ is not primitive, then there exists a subgroup of $\gal(F_s/\qq)$ of order 240 or 360 containing $\gal(F_s/F)$ as a maximal subgroup. The maximal subgroups of $S_6$ are of order 48, 72, 120, 
    or 360\footnote{https://groupprops.subwiki.org/wiki/Subgroup\_structure\_of\_symmetric\_group:S6}. In particular, there is no subgroup of order 240. The only subgroup of $S_6$ of order 360 is $A_6$. But then $\gal(F_s/F)$ would be a subgroup of order 120 inside $A_6 = \gal(F_s/F_s^{A_6})$; but $A_6$ does not have a subgroup of order 120. Thus, $F$ is primitive.

\subsubsection{$G = S_4$}
If $\gal(F_s/\qq)\cong S_4$, then $H:=\gal(F_s/F)$ corresponds to an order 4 non-normal subgroup of $S_4$. From the lattice of subgroups of $S_4$ (App.\,\ref{sec:app}, Fig.\,\ref{fig:lattice-s4}), we see 
that $H$ is $C_4^j$ or $V_4^j$ for some $j=1,2,3$. ($H$ cannot be $V_4^n$ since it is normal in $S_4$.) The only proper subgroup containing $H$ is the corresponding $D_4^j$. Hence 
$F$ has a unique necessarily cubic subfield--denote this as $F_c$. We will show that $F_c$ is not totally real, from which we conclude that $F$ is of 
{\bf TR}-type. To proceed, we first need the following lemma whose proof is an easy exercise. 

\begin{lem}
\label{embed}
 Suppose $E/F/\qq$ is a tower of finite extensions and $E/\qq$ is Galois with Galois group $G$. Let $H=\gal(E/F)$. Then $\Hm(F,E)$ is in bijection with the set of cosets $G/H$. Furthermore, the bijection is $G$-equivariant.
\end{lem}

Let $\Sigma_{F_c}=\{\sigma_1,\sigma_2,\sigma_3\}$ be the set of embeddings of $F_c$ into $\cc$. Suppose, by the way of seeking a contradiction,  that $F_c$ is totally real. If $\mf{c}\in\aut(\cc)$ denotes complex conjugation, we should have $\mf{c}\circ \sigma_k =\sigma_k$ for all $k=1,2,3$. Note that $\mf{c}|_{F_s}$ can be regarded as an element of order 2 in $S_4$. By the lemma above, we have a $S_4$-equivariant bijection $\Sigma_{F_c}\cong \gal(F_s/\qq)/\gal(F_s/F_c)=S_4/D_4^j$. Therefore, to say that $\mf{c}\circ \sigma_k =\sigma_k$ for all $k=1,2,3$ is equivalent to 
saying $\mf{c}(xD_4^j)=xD_4^j$ for all $x\in S_4$. This means that 
$\mf{c}\in\bigcap_{x\in S_4}xD_4^jx^{-1}=\bigcap_{j=1,2,3}D_4^j=V_4^n=\left\{1,(12)(34),(13)(24),(14)(23)\right\}.$ Hence, 
$$
\mf{c}\in \left\{(12)(34),(13)(24),(14)(23)\right\}.
$$
Given that $F$ is totally imaginary, $\mf{c}$ exchanges pairs of conjugate-embeddings, or equivalently, exchanges pairs of cosets in $S_4/H$. 
We look at the action of the possible choices of $\mf{c}$ on different possibilities of $H$.

\smallskip
First suppose that $H=C_4^j$. Assume, without loss of generality that 
\[
H=\{1,(1234),(13)(24),(1432)\}.
\]
\begin{itemize}
\item If $\mf{c}=(12)(34)$, we have $\mf{c}((14)H)=(14)H$ since $(14)(13)(24)(14)=(12)(34).$
\item If $\mf{c}=(13)(24)$, we have $\mf{c}H=H$.
\item If $\mf{c}=(14)(23)$, we have $\mf{c}((12)H)=(12)H$ since $(12)(13)(24)(12)=(14)(23)$.
\end{itemize}

\smallskip
Now suppose that $H=V_4^j$. Assume, without loss of generality that 
\[
H=\{1,(12),(34),(12)(34)\}.
\]
\begin{itemize}
\item If $\mf{c}=(12)(34)$, then $\mf{c}H=H$. 
\item If $\mf{c}=(13)(24)$, we have $\mf{c}((23)H)=(23)H$ since $(23)(12)(34)(23)=(13)(24)$.
\item  If $\mf{c}=(14)(23)$, we have $\mf{c}((24)H)=(24)H$ since $(24)(12)(34)(24)=(14)(23)$.
\end{itemize}
In all cases $\mf{c}$ fixes a coset and so the corresponding embedding of $F$ is real--a contradiction since $F$ is totally imaginary. 
Therefore the unique cubic subfield $F_c$ can only be of mixed signature.

\subsubsection{$G = D_6$}
\label{D6}
If $\gal(F_s/\qq)\cong D_6$, then $H = \gal(F_s/F)$ is an order 2 subgroup of $D_6$. Since $F/\qq$ is not Galois, we rule out the case of $H =\langle r^3\rangle$. Any subfield $F'\subset F$ corresponds to either an order 4 or order 6 subgroup of $D_6$ containing $H$. By the subgroup lattice of $D_6$ (App.\,\ref{sec:app}, Fig. \,\ref{fig:lattice-d6}), every order 2 subgroup (that is not $\langle r^3\rangle$) of $D_6$ is contained in a unique subgroup of order 6 and a unique subgroup of order 4. Thus, $F$ has a unique imaginary quadratic and a unique cubic subfield. Since the order 4 subgroups are not normal, the cubic subfield can be totally real or mixed signature.

We want to determine when the cubic subfield is totally real (resp., mixed signature). Let $H_c$ denote the order 4 subgroup corresponding to the cubic subfield $F_c=F^{H_c}$. Let $\mf{c}$ denote complex conjugation whose restriction to $F_s$ is an element of order 2 in $D_6$ that we also denote $\mf{c}$. We know that $\mf{c}$ acts on the set of embeddings $\Sigma_{F_c}$ of $F_c$ into $\cc$. Lemma \ref{embed} implies that $\mf{c}$ acts on the set $G/H_c$. We can determine whether $F_c$ is totally real by analyzing the action of $\mf{c}$ on $G/H_c$. In particular, $F_c$ is totally real if and only if  $\mf{c}$ acts trivially on $G/H_c$, which is equivalent to  $ \mf{c}(gH_c)=gH_c$ for all $g\in G$. This gives us the condition
$ \mf{c}\in \bigcap_{g\in G} gH_cg^{-1}.$ 
By the same argument, if $H_q$ denote the order 6 subgroup corresponding to the imaginary quadratic subfield $F_q=F^{H_q}$, then $\mf{c}$ should exchange pairs of cosets in $G/H_q$. Similarly, since $F$ is totally imaginary, $\mf{c}$ should exchange pairs of cosets in $G/H$. From these two cases, we know that $\mf{c}\not\in H_q$ and $\mf{c}\not\in H$. To summarize, $\mf{c}$
should satisfy:
$$
\mf{c}\in \bigcap_{g\in G} gH_cg^{-1}, \quad \mf{c}\not\in H_q, \quad \mf{c}\not\in H. \quad \quad \quad \quad (*)
$$

With these properties, we can now enumerate the possible choices for $\mf{c}$ given different choices of $H$. From the symmetry in the subgroup lattice of $D_6$ (App.\,\ref{sec:app}, Fig.\,\ref{fig:lattice-d6}), it suffices to consider the cases when $H\in\{\langle s\rangle,\,\langle sr^2\rangle\,\langle sr^4\rangle\}$.
When $H=\langle s\rangle$, we know $H_c=\langle s,r^3\rangle$ and $H_q=\langle s,r^2\rangle$. Since $\mf{c}$ is of order 2 and $\mf{c}\not\in H,H_q$, one can only have $\mf{c}\in\{r^3,sr^5,sr^3,sr\}$. It is easy to verify that only $\mf{c}=r^3$ satisfies $(*)$. The other two cases of $H$ are identical to the case $H=\langle s\rangle$. We can therefore conclude that $F_c$ is totally real if and only if $\mf{c}$ is central (since $r^3$ is the only order 2 central element in $D_6$). The following two examples show that both cases--$F_c$ being totally real or mixed signature can happen. 

\begin{enumerate}
\item The number field with label 6.0.309123.1 on LMFDB is a totally imaginary sextic field with Galois group $D_6$; it is a CM-field; its unique cubic subfield 
$F_c$ with label 3.3.321.1 is totally real ($\mf{c}$ is necessarily central); the unique quadratic subfield $F_q$ is $\qq(\sqrt{-3}).$ 
\item The number field with label 6.0.14283.1 on LMFDB is a totally imaginary sextic field with Galois group $D_6$ is not a CM-field; its unique cubic subfield $F_c$ with label 
3.1.23.1 is of mixed signature; the unique quadratic subfield is $\qq(\sqrt{-3}).$ Hence $F$ is of {\bf CM}-type but not a CM-field. 
\end{enumerate}

In the spirit of this article, one may formulate the na\"ive problem: 
\begin{prob}
\label{prob:D6}
When ordered by discriminant, determine what percentage of sextic totally imaginary fields with Galois group $D_6$ are not CM-fields. 
\end{prob}
Following Altug--Shankar--Varma--Wilson \cite{ali-s-v-w} one can ask for variations of the above problem exploiting the structure of $D_6$. 
Going modulo the center of $D_6$ we get a surjection $D_6 \to S_3$. Consider the unique two-dimensional irreducible representation of $S_3$ and inflate it back to $D_6$, which gives us  
the unique two-dimensional irreducible representation of $D_6$ with trivial central character. We can order the totally imaginary sextic fields with Galois group $D_6 = \gal(F_s/\qq)$ using the Artin conductor of this particular two-dimensional representation and then ask the question of determining what percentage of such fields are CM-fields.

\subsubsection{$G = S_3 \times C_3$}
 If $\gal(F_s/\qq)\cong S_3\times C_3$, then $H = \gal(F_s/F)$ corresponds to an order 3 subgroup of $S_3\times C_3$. Consider the image of $H$ under the canonical projections
$$
\xymatrix{
 & S_3\times C_3 \ar[rd]^{\pi_2} \ar[ld]_{\pi_1} & \\
 S_3 & & C_3 
}
$$ 
Clearly, $|\pi_2(H)|$ divides $3.$ Suppose $|\pi_2(H)|=1$, then we know that $H\subset\ker(\pi_2)=S_3\times\{1\}$. The unique subgroup of order $3$ in $S_3$ is $A_3$, thus, $H\cong A_3\times \{1\}$, which is normal in $S_3\times C_3$. This would force $F$ to be Galois over $\qq$, which is a contradiction.
Hence, $|\pi_2(H)|=3$. Now looking at the order of $\pi_1(H)$ we only have two cases:

(i) $|\pi_1(H)|=1$. Then $H = \ker(\pi_1)=\{1\}\times C_3$ which is normal in $S_3\times C_3$, whence $F$ is Galois over $\qq$, which is a contradiction.

(ii) $|\pi_1(H)|=3$. This implies $H\cong \pi_1(H) \cong A_3$. Furthermore, there is a surjective (hence bijective) map $f:\pi_1(H)=A_3\to\pi_2(H)=C_3$ such that
$H=\{(x,f(x))\mid x\in A_3\}.$ In fact, $f$ is an isomorphism. Also, there is a unique intermediate subgroup: $H\subset A_3\times C_3\subset_2 S_3\times C_3.$ 
Therefore, $F_s^{A_3\times C_3}$ is a quadratic subfield of $F$ that is necessarily imaginary. Hence $F$ is of {\bf CM}-type but not a CM-field.

\subsubsection{$G = A_4 \times C_2$}
If $\gal(F_s/\qq)\cong A_4\times C_2$, then $\gal(F_s/F)$ corresponds to an order 4 subgroup of $A_4\times C_2$. To classify such order 4 subgroups, consider the projections 
$$
\xymatrix{
 & A_4\times C_2 \ar[rd]^{\pi_2} \ar[ld]_{\pi_1} & \\
 A_4 & & C_2 
}
$$ 
The order of $\pi_1(H)$ divides $4$ and the order of $\pi_2(H)$ divides $2$. If $\pi_2(H)$ is trivial then $H\subset\ker(\pi_2)=A_4\times\{1\}$; by the subgroup lattice, the unique subgroup of order $4$ in $A_4$ is $V_4$, thus, $H\cong V_4\times \{1\}$, which is normal in $A_4\times C_2$; this would force $F$ to be Galois over $\qq$, which is a contradiction. Hence $\pi_2(H) = C_2$. 
By order consideration, we can only have two cases:

(i) $|\pi_1(H)|=2$. By the exact sequence $1\to H\cap\ker(\pi_1)\to H\xrightarrow{\pi_1} \pi_1(H)\to 1$, 
we know that $|H\cap\ker(\pi_1)|=|H|/|\pi_1(H)|=2$, which implies $\ker(\pi_1)=\{1\}\times C_2\leq H$. Since $\{1\}\times C_2$ is normal in $A_4\times C_2$, we would have $F_s^{\{1\}\times C_2}/\qq$ is Galois, which contradicts the fact that $F_s$ is the smallest Galois extension over $\qq$ containing $F$.

(ii) $|\pi_1(H)|=4$. This implies $H\cong \pi_1(H) \cong V_4$. There exists a surjective map $f:\pi_1(H)=V_4\to\pi_2(H)=C_2$ such that
$H=\{(x,f(x))\mid x\in V_4\}.$ 
Notice that $H\subset V_4\times C_2\lhd_3 A_4\times C_2.$ 
Therefore, $F_s^{V_4\times C_2}$ is a cubic subfield of $F$ which is Galois. This forces the cubic subfield to be totally real. Hence, $F$ is a CM field.

\subsubsection{$G=S_3\times S_3$}
If $\gal(F_s/\qq)\cong S_3\times S_3$, then $H = \gal(F_s/F)$ is a subgroup of $S_3\times S_3$ or order $6$. To classify all such subgroups. Consider the projections
$$
\xymatrix{
 & S_3 \times S_3 \ar[rd]^{\pi_2} \ar[ld]_{\pi_1} & \\
 S_3 & & S_3
}
$$ 
We must have the orders of $\pi_1(H)$ and $\pi_2(H)$ to be factors of $6$. 

Suppose $|\pi_2(H)|=1$. Thus, $H\subset\ker(\pi_2)$, which forces $H\cong S_3\times\{1\}$ by order consideration; such a subgroup is normal in $S_3\times S_3$, forcing 
$F$ to be Galois over $\qq$, which is a contradiction.

Suppose $|\pi_2(H)|=2$, then by  the  exact sequence $1\to H\cap\ker(\pi_2)\to H\xrightarrow{\pi_2} \pi_2(H)\to 1$, 
we know that $|H\cap\ker(\pi_2)|=3$, which implies $C_3\times\{1\}\leq H$. Since $C_3\times\{1\}$ is normal in $S_3\times S_3$, we then have $F_s^{C_3\times \{1\}}/\qq$ is Galois, which  contradicts $F_s$ being the smallest Galois extension over $\qq$ containing $F$.

Suppose $|\pi_2(H)|=3$, then by the same exact sequence $|H\cap\ker(\pi_2)|=2$, which implies $C_2\times\{1\}\subset H$. This implies that $H=C_2\times C_3$ or 
$H=\{(x,f(x))\mid x\in S_3\}$ for $f$ a surjective map $f:\pi_1(H)=S_3\to\pi_2(H)=C_3$; the latter is ruled out as there is no nontrivial morphism from $S_3$ to $C_3$, and 
the former is ruled out because it contains $C_3\times\{1\}$ arguing then as in the previous paragraph. 

Suppose $|\pi_2(H)|=6$. We will rule out $\{1\}\times S_3$ as above. The only remaining possible $H$ is of the form 
\[H=\{(x,f(x))\mid x\in S_3\},\]
for $f$ an automorphism of $S_3$. Consider the following homomorphism
\[\phi:S_3\times S_3\to\{\pm 1\},\quad (x,y)\mapsto \sgn(x)\sgn(y)\]
and let $K=\ker(\phi)$. Since $\phi$ is surjective, we note that $K$ is a subgroup of order $18$ containing $H$ for all $f$. Thus, $F_s^K$ is a quadratic subfield of $F$. By order consideration, $H$ is maximal in $K$. Thus, any other subgroup $L$ of $S_3\times S_3$ containing $H$ has to contain an element of the form $(x,y)$, where $\sgn(x)\neq \sgn(y)$. Since the $H$'s are isomorphic for different choices of $f$, we will work with $f=1_{S_3}$ for simplicity, in which case $H$ is the usual diagonal $S_3$ in the product $S_3 \times S_3.$ 
We claim that $|L|\neq 12$. Suppose otherwise, then for all $l\in L$, we should have $l^2\in H$. By our observation, $L$ needs to contain a $(1-cycle,2-cycle)$ or $(2-cycle,3-cycle)$ (or exchanging the first and second copy). It is clear that $(2-cycle,3-cycle)^2\not\in H$. Without loss of generality we can take the $(1-cycle,2-cycle)$ to be $(1,(12))$. Multiplying by $((23),(23))$ on the right, we get $((23),(231))$, whose square does not lie in the diagonal. Thus, we arrive at a contradiction and $|L|\neq 12$. We conclude that $K$ is the only subgroup containing $H$. 
Hence, $F$ contains a unique quadratic subfield, which is necessarily imaginary, whence, $F$ is of {\bf CM}-type but not a CM-field.

\subsubsection{$G = S_4 \times C_2$}
If $\gal(F_s/\qq)\cong S_4\times C_2$, then $H:=\gal(F_s/F)$ corresponds to an order 8 subgroup. To classify all such subgroups. Consider the projections
$$
\xymatrix{
 & S_4 \times C_2 \ar[rd]^{\pi_2} \ar[ld]_{\pi_1} & \\
 S_4 & & C_2
}
$$ 
The order of $\pi_1(H)$ must divide $8$.

Suppose $|\pi_1(H)|=1$, then $H\subset\ker(\pi_1)=\{1\}\times C_2$, which is impossible.

Suppose $|\pi_1(H)|=2$, then by  the exact sequence $1\to H\cap\ker(\pi_1)\to H\xrightarrow{\pi_1} \pi_1(H)\to 1$, 
we get $|H\cap\ker(\pi_1)|=|H|/|\pi_1(H)|=4$, which is impossible.

Suppose $|\pi_1(H)|=4$, then $|H\cap\ker(\pi_1)|=2$, or equivalently $\ker(\pi_1)=\{1\}\times C_2\leq H$. Since $\{1\}\times C_2$ is normal in $S_4\times C_2$, we would have $F_s^{\{1\}\times C_2}/\qq$ is Galois contradicting $F_s$ being the smallest Galois extension over $\qq$ containing $F$.

Suppose $|\pi_1(H)|=8$. Consider two sub-cases depending on the order of $\pi_2(H)$: 

Suppose that $|\pi_1(H)|=8$ and $|\pi_2(H)|=1$, then $H\cong D_4\times\{1\}$ since $D_4$ is the unique order 8 subgroup of $S_4$. By the group lattice of $S_4$ 
(App.\,\ref{sec:app}, Fig.\,\ref{fig:lattice-s4}) we see that $D_4$ contains the unique Klein-4 subgroup that is normal in $S_4$ denoted by $V_4^n$. But then $F_s^{V_4^n\times\{1\}}/\qq$ is Galois, contradicting  
$F_s$ being the smallest Galois extension over $\qq$ containing $F$.

Suppose now that $|\pi_1(H)|=8$ and $|\pi_2(H)|=2.$ Then  $\pi_1|_H$ is injective, hence $H\cong \pi_1(H)\cong D_4$. 
From the group lattice of $S_4$ (App.\,\ref{sec:app}, Fig.\,\ref{fig:lattice-s4}) we see that $S_4$ contains three copies of $D_4$ denoted $D_4^j$ for $j=1,2,3$. 
Each $D_4^j$ contains $V_4^n$, a copy of $C_4$ denoted by $C_4^j$, and a non-normal copy of $V_4$ denoted by $V_4^j$.
Assume without loss of generality that $\pi_1(H)\cong D_4^j$ for a particular $j$. Since $|\pi_2(H)|=2$, there exists a surjective map $f:\pi_1(H)\cong D_4^j\to\pi_2(H)=C_2$ such that 
\[H=\left\{(x,f(x))\mid x\in D_4^j\right\}.\]
Observe that since $\pi_1$ is injective, $f$ needs to be a function. Since $(1,1)\in H$, we necessarily have $f(1)=1$. Moreover, $(x,f(x))(y,f(y))=(xy,f(x)f(y))$ combined with the fact that $f$ is a function shows that $f(xy)=f(x)f(y)$. Therefore, $f$ is a group homomorphism. Thus, $H$ is completely determined by the datum $(D_4^j,f)$. Since $\ker(f)$ is normal in $D_4^j$ and of order 4, it could be $V_4^n$, $V_4^j$, or $C_4^j$. If $\ker(f)=V_4^n$, then $V_4^n\times\{1\}\leq H$ is normal in $S_4\times C_2$ and we have a contradiction as before. Thus, we can only have $\ker(f)=V_4^j$ or $\ker(f)=C_4^j$. In each case, the kernel determines $f$ completely. Observe that 
\[H=(D_4^j,f)\subset_2 D_4^j\times C_2\subset_3 S_4\times C_2,\]
where the subscript indicates the index. Therefore, $D_4^j\times C_2$ corresponds to a cubic subfield of $F$. Since $D_4^j\times C_2$ is not normal, the cubic subfield can be totally real or mixed signature. In fact, this is the unique cubic subfield in $F$, or equivalently, $D_4^j\times C_2$ is the unique index 3 subgroup of $S_4\times C_2$ containing $H$. To see this, consider a subgroup $K$ such that
\[H\subset_2 K \subset_3 S_4\times C_2.\]
Since $|K|=16$, we necessarily have $\pi_1(K)\cong D_4^j$ and $\pi_2(K)=C_2$. Let $(x,y)\in K\setminus H$, then $(x,-y)\in H$, which implies $(x^{-1},-y)\in H$. Thus, $(x,y)(x^{-1},-y)=(1,-1)\in K$ and we conclude that $\{1\}\times C_2\leq K$. Consequently, for any $(x,y)\in K$, we have $(1,y)\in K$ and $(x,y)(1,y)=(x,1)\in K$. Thus, $D_4^j\times\{1\}\leq K$, which forces $K\cong D_4^j\times C_2$.

Suppose there exists $H'$ such that $H\subset_3 H'$. Then we must have $\pi_1(H')=S_4$ and $\pi_2(H')=C_2$. In other words, $H'$ is also of the form $(S_4,f')$, where $f' |_{D_4^j} =f$. Since $|\ker(f')|=12$, we know $\ker(f')=A_4$. Since $\ker(f) \subset \ker(f')$ but also since neither $C_4^j$ nor $V_4^j$ is contained in $A_4$, we get a contradiction. Thus, $F$ has no quadratic subfields.

We want to determine when the unique cubic subfield is totally real or of mixed signature. Following the case of $D_6$ in \S \ref{D6}, we need to examine the action of complex conjugation $\mf{c}$ on the set of embeddings of the cubic subfield into $\cc$. Note that $\mf{c}$ should be of the form 
\[\text{$(e_1,e_2)$, where $e_1\in\{C_2^{11},C_2^{21}, C_2^{31}, C_2^1,\cdots, C_2^6\}$ and $e_2\in C_2$;}\quad (e_1,e_2)\neq (1,1).\]
 Since $F$ is totally imaginary, the action of $\mf{c}$ on the set of cosets $gH$ is nontrivial (and in particular exchanges pairs of cosets in $G/H$). This implies 
\[\mf{c}\not\in H \quad \quad \quad (*).\] 
Let $H_c = D_4^j\times C_2$ be the subgroup corresponding to the cubic subfield, then the analysis in \ref{D6} tells us that
\[\text{$F_c$ is totally real }\iff \mf{c}\in\bigcap_{g\in G} gH_cg^{-1}\Leftrightarrow  \mf{c}\in V_4^n\times C_2\quad \quad \quad (**).\]

First, consider the case when $\ker(f)=C_4^j$. Without loss of generality, take $j=1$. Among all the pairs $(e_1,e_2)$, we know from the information of $\ker(f)$ combined with the group lattice that those in the following subsets (excluding the element $(1,1)$) satisfy $(*)$:
\begin{enumerate}[label=(\roman*)]
    \item $C_2^{11}\times \{-1\}$;
    \item $C_2^{j'1}\times \{1\}$, $j'=2,3$;
    \item $C_2^i\times\{1\}$, $i=1,2$;
    \item $C_2^i\times\{\pm 1\}$, $i=3,\dots, 6$.
\end{enumerate}
Among these, case (i) and case (ii) satisfy $(**)$. Thus, when $\mf{c}$ is in case (i) or (ii), $F_c$ is totally real and $F$ is a $CM$-field. 
If $\mf{c}$ is in case (iii) or (iv), $F_c$ is of mixed signature and $F$ is of {\bf TR}-type.

Now consider the case when $\ker(f)=V_4^j$. Without loss of generality we can take $j=1$. Among all the pairs $(e_1,e_2)$, we know from the information of $\ker(f)$ combined with the group lattice that those in the following subsets (excluding the element $(1,1)$) satisfies $(*)$:
\begin{enumerate}[label=(\roman*)]
    \item $C_2^{11}\times \{-1\}$;
    \item $C_2^{j'1}\times \{1\}$, $j'=2,3$;
    \item $C_2^i\times\{-1\}$, $i=1,2$;
    \item $C_2^i\times\{\pm 1\}$, $i=3,\dots, 6$.
\end{enumerate}
Among these, case (i) and case (ii) satisfy $(**)$. Thus, when $\mf{c}$ is in case (i) or (ii), $F_c$ is totally real and $F$ is $CM$. If $\mf{c}$ is in case (iii) or (iv), $F_c$ is of mixed signature and $F$ is of {\bf TR}-type.

\smallskip

To summarize in general for $j =1,2,3$, if $\mf{c}=(e_1,-1)$ with $e_1\in C_2^{j1}$ or $\mf{c}=(e_1,1)$ with $e_1\in C_2^{j'1}$ ($j'\neq j$), then $F$ is CM. Otherwise, $F$ is of \textbf{TR}-type.

\smallskip

One may ask for the analogue of Prob.\,\ref{prob:D6} for $S_4 \times C_2$-extensions:
\begin{prob}
\label{prob:S4xC2}
When ordered by discriminant, determine what percentage of sextic totally imaginary fields with Galois group $S_4 \times C_2$ are CM-fields. 
\end{prob}

\subsubsection{$G = C_3^2 \rtimes D_4$}
If $\gal(F_s/\qq)\cong C_3^2\rtimes D_4$, then $H=\gal(F_s/F)$ corresponds to an order 12 subgroup of $C_3^2\rtimes D_4$. We first exhibit a particular presentation of $C_3^2 \rtimes D_4$ in $S_6$. Denote the $6$ letters to permute as $\{1,2,3,a,b,c\}.$ 
Let 
\[C_3^2=\langle(123)\rangle\times\langle(abc)\rangle \qquad 
D_4=\langle(1a2b)(3,c)\rangle\rtimes\langle(ab)\rangle\]
It is easily seen that $r=(1a2b)(3,c)$ is the rotation element and $s=(ab)$ is the reflection element. The action of $D_4$ on $C_3^2$ can be explicated as:
\[r(123)r^{-1}=(1a2b)(3c)(123)(1b2a)(3c)=(abc),\qquad s(123)s=(123),\]
\[r(abc)r^{-1}=(1a2b)(3c)(abc)(1b2a)(3c)=(132),\qquad s(abc)s =(acb).\]
In order to classify all order 12 subgroups of $G$, we need the following seemingly well-known result\footnote{https://kconrad.math.uconn.edu/blurbs/grouptheory/group12.pdf}
classifying groups of order $12$: 

\begin{prop}
    Every group of order 12 is isomorphic to one of the following:
    \begin{enumerate}[label=(\roman*)]
        \item $\zz/12\cong \zz/3\times\zz/4$, 
        \item $(\zz/2)^2\times\zz/3$, 
        \item $(\zz/2)^2\rtimes\zz/3$ (one isomorphism class), 
        \item $\zz/3\rtimes (\zz/2)^2$ (one isomorphism class), 
        \item $\zz/3\rtimes\zz/4$ (one isomorphism class). 
    \end{enumerate}
\end{prop}
Since $D_4$ has no element of order 3, an order 3 element of $G$ can only come from the $C_3^2$ component. Thus, type (ii) and (iii) cannot exist in the ambient group. Moreover, since an element of order 2 or 4 can only come from the $D_4$ component and the action of $D_4$ on $C_3^2$ is not trivial, $H$ can only be of type (iv) or (v).

Thus, we know that $H$ is of the form $A\rtimes K$ with $A\leq C_3^2$, $|A|=3$ and $K\leq D_4$, $|K|=4$. The choices for $A$ are:
$C_3\times\{1\},$ $\{1\}\times C_3,$ or $\Delta_{C_3,f}= \{(x,f(x))\mid x\in C_3\}$ for some isomorphism $f:C_3\to C_3.$  
There are two distinct choices of $f$. Suppose $f:C_3=\langle x\rangle\to C_3=\langle y\rangle$. Then there exist $f_1$ and $f_2$ characterised by 
$f_1(x)=y$ and $f_2(x)=y^2$. Therefore, in all, there 
are four choices for $A$: 
\[
C_3\times\{1\}, \quad \{1\}\times C_3, \quad \Delta_{C_3,f_1}, \quad \Delta_{C_3,f_2}.
\]
It is easy to that possible choices for $K$ are: 
\[\langle r\rangle,\quad V_1 = \langle r^2\rangle\times\langle s\rangle,\quad V_2 = \langle r^2\rangle\times\langle sr\rangle.\]

We want to see if various $K$ acts on the different choices of $A$ by conjugation. By explicit computation using the action of $D_4$ explicated above, we have obtain the following list of feasible $K$ for a given $A$: 
\begin{equation}
\label{eqn:feasible-a-k}
\begin{tabular}{|c|c|}
    \hline
    $A$ & $K$ \\
    \hline
    $C_3\times \{1\}$ & $V_1$\\
    \hline
    $\{1\}\times C_3$ & $V_1$\\
    \hline
    $\Delta_{C_3,f_1}$ & $V_2$\\
    \hline
    $\Delta_{C_3,f_2}$ & $V_2$\\
    \hline
\end{tabular}
\end{equation}
Let $H=A\rtimes K$ with one of the pair $(A,K)$ in the table. We want to find the possible subgroups $L$ such that $H\subsetneq L \subsetneq G=C_3^2\rtimes D_4$.

First note that we can have $L=C_3^2\rtimes K$, which cuts out a quadratic subfield of $F$ since $[L:H]=3$; this quadratic subfield is necessarily imaginary and so $F$ is 
of {\bf CM}-type. 

Next, we claim that $G$ does not have a proper subgroup $L$ with $[L:H]=2$, i.e., $F$ does not have any cubic subfields. 
Suppose otherwise, then we must have the following commutative diagram. There exist sections in the top and bottom rows because $G$ and $H$ are semidirect products.
\begin{figure}[H]
    \centering
    \begin{tikzcd}
        1\ar[r] & C_3^2\ar[r] & G\ar[r,"\pi_2"] & D_4\ar[r]\ar[l,bend left=20,"s"] & 1\\
        1\ar[r] & C_3^2\cap L\ar[r]\ar[d,equal]\ar[u,hook] & L\ar[r,"\pi_2"]\ar[u,hook] & D_4\ar[r]\ar[u,equal] & 1\\
        1\ar[r] & A\ar[r] & H\ar[r,"\pi_2"]\ar[u,hook] & K\ar[r]\ar[l,bend left=20,"s"]\ar[u,hook] & 1
    \end{tikzcd}
\end{figure}

 We want to show that $L$ is also a semidirect product, i.e., there exists a section $s':D_4\hookrightarrow L$ compatible with $s:D_4\to G$. Suppose, by the way of seeking a contradiction, that such an $s'$ does not exist. We will separate the analysis into two cases based on the choice of $K$.

First consider $K=V_1=\langle s,r^2\rangle$, then $A= C_3\times\{1\}$ or $\{1\}\times C_3$ based on the feasible pairs of $(A,K)$ as in \eqref{eqn:feasible-a-k}. To say that $s'$ does not exist is to say that $s(D_4)\not\subset L$. Since from the commutative diagram we know that $s(V_1)\subset L$, we must have $(1,sr)\not\in L$. Since $L\twoheadrightarrow D_4$, there exists some element $(x,sr)\in L$ with $x\neq 1$ and $x\in A$. Moreover, we must have $(x,sr)^2\in H$ by the assumption on index. We check which choices of $x$ satisfy these criteria. For simplicity let $\alpha=(123)$ and $\beta=(abc)$.
$$
\begin{tabular}{|c|c|c|}
    \hline
    choice of $(x,sr)$ & $(x,sr)^2$& $(x,sr)^2\in H$? \\
    \hline
    $(\alpha,sr)$ & $(\alpha,sr)^2=(\alpha sr\alpha r^{-1}s^{-1},1)=(\alpha\beta^2,1)$& no\\
    \hline
    $(\alpha^2,sr)$ & $(\alpha^2,sr)^2=(\alpha^2 sr\alpha^2 r^{-1}s^{-1},1)=(\alpha^2\beta,1)$& no\\
    \hline
    $(\beta,sr)$ & $(\beta,sr)^2=(\beta sr\beta r^{-1}s^{-1},1)=(\beta\alpha^2,1)$& no\\
    \hline
    $(\beta^2,sr)$ & $(\beta^2,sr)^2=(\beta^2 sr\beta^2 r^{-1}s^{-1},1)=(\beta^2\alpha,1)$& no\\
    \hline
\end{tabular}
$$
Thus, the element $(x,sr)\in L$ as required cannot exist. Therefore, $s'$ exists and $L=A\rtimes D_4$.

Next consider $K=V_2=\langle sr,r^2\rangle$, then $A= \Delta_{C_3,f_1}=\langle \alpha\beta\rangle$ or $\Delta_{C_3,f_2}=\langle \alpha\beta^2\rangle$ based on the feasible pairs of $(A,K)$ 
as in \eqref{eqn:feasible-a-k}. To say that $s'$ does not exist is to say that $s(D_4)\not\subset L$. Since from the commutative diagram we know that $s(V_2)\subset L$, we must have $(1,s)\not\in L$. Since $L\twoheadrightarrow D_4$, there exists some element $(x,s)\in L$ with $x\neq 1$ and $x\in A$. Moreover, we must have $(x,s)^2\in H$ by the assumption on index. We check which choices of $x$ satisfy these criteria.
$$
\begin{tabular}{|c|c|c|}
    \hline
    choice of $(x,s)$ & $(x,s)^2$& $(x,s)^2\in H$? \\
    \hline
    $(\alpha\beta,s)$ & $(\alpha\beta,s)^2=(\alpha\beta s\alpha\beta s^{-1},1)=(\alpha^2,1)$& no\\
    \hline
    $(\alpha^2\beta^2,s)$ & $(\alpha^2\beta^2,s)^2=(\alpha^2\beta^2 s\alpha^2\beta^2s^{-1},1)=(\alpha,1)$& no\\
    \hline
    $(\alpha\beta^2,s)$ & $(\alpha\beta^2,s)^2=(\alpha\beta^2 s\alpha\beta^2s^{-1},1)=(\alpha^2,1)$& no\\
    \hline
    $(\alpha^2\beta,s)$ & $(\alpha^2\beta,s)^2=(\alpha^2\beta s\alpha^2\beta s^{-1},1)=(\alpha,1)$& no\\
    \hline
\end{tabular}
$$
Thus, the element $(x,s)\in L$ as required cannot exist. Therefore, $s'$ exists and $L=A\rtimes D_4$.

From our analysis above, any subgroup $L$ with $[L:H]=2$ must be of the form $L=A\rtimes D_4$. However, from the table we see that $D_4$ cannot stabilize any choice of $A$ since 
$\langle r\rangle$ does not act on any $A$ by conjugation. Thus, $F$ does not have any cubic subfields.

\smallskip
In conclusion, any totally imaginary sextic field $F$ with $\gal(F_s/\qq) = C_3^2 \rtimes D_4$ has a unique subfield which is imaginary quadratic and hence $F$ is of {\bf CM}-type but not 
a CM-field. 

\medskip
\subsection{To summarise:} The following table is a summary of the sextic totally imaginary fields: 

$$
\begin{tabular}{|l|p{4cm}|c|c|c|}
    \hline
    Galois group & Subfield & CM & \textbf{CM}-type 
    but not CM & \textbf{TR}-type\\
    \hline
    $C_6$ & imaginary quadratic, totally real cubic & $\bullet$ & &\\
    \hline
    $S_3$ & imaginary quadratic, mixed signature cubic &  &$\bullet$ &\\
     \hline
    $D_6, \ \ {\rm Case \ (i)}$ & imaginary quadratic, totally real cubic & $\bullet$ & &\\
    \hline
    $D_6, \ \ {\rm Case \ (ii)}$ & imaginary quadratic, mixed signature cubic &  & $\bullet$&\\
    \hline
    $A_4\times C_2$ & totally real cubic & $\bullet$ & &\\
    \hline
    $S_4\times C_2, \ \ {\rm Case \ (i)}$ & totally real cubic & {$\bullet$} & &\\
    \hline
    $S_4\times C_2, \ \ {\rm Case \ (ii)}$ & mixed signature cubic &  & &{$\bullet$}\\
     \hline
    $S_3\times C_3$  & imaginary quadratic &  & $\bullet$&\\
    \hline
    $C_3^2\rtimes D_4$ & imaginary quadratic &  &$\bullet$ &\\
    \hline
    $S_3\times S_3$ & imaginary quadratic &  &$\bullet$ &\\
       \hline
    $S_4$ & mixed signature cubic &  & &$\bullet$\\
    \hline
    $S_5\cong PGL_2(\fcy{F}_5)$ & None &  & &$\bullet$\\
    \hline
    $S_6$ & None &  & &$\bullet$\\
    \hline
\end{tabular}
$$

\newpage
\appendix
\section{Lattice of subgroups for $D_4$, $D_6$, and $S_4$}
\label{sec:app}
\begin{figure}[h]
    \centering
    \begin{tikzpicture}[node distance=2cm,line width=0.5pt]
\node(D4) at (0,0)     {$D_4$};
\node(V4)       [below right =1cm and 2cm of D4] {$V_4^N=\langle r^3s, r^2\rangle$};
\node(V4N)      [below left =1cm and 2cm of D4]  {$V_4^N=\langle r^2, s\rangle$};
\node(C4)      [below = 1cm of D4]       {$C_4=\langle r\rangle$};
\node(C21)      [below left =1cm and 1cm of V4N]       {$C_2=\langle r^2s\rangle$};
\node(C22)      [below =1cm of V4N]       {$C_2=\langle s\rangle$};
\node(C2)      [below =1cm of C4]       {$C_2=\langle r^2\rangle$};
\node(C23)      [below =1cm of V4]      {$C_2=\langle r^3s\rangle$};
\node(C24)      [below right =1cm and 1cm of V4]      {$C_2=\langle rs\rangle$};
\node(1)        [below=1cm of C2]   {$\left\{1\right\}$};
\draw(D4)       -- (V4);
\draw(D4)       -- (C4);
\draw(D4)       -- (V4N);
\draw(V4N)      -- (C21);
\draw(V4N)      -- (C22);
\draw(V4N)      -- (C2);
\draw(C4)      -- (C2);
\draw(V4)      -- (C2);
\draw(V4)      -- (C23);
\draw(V4)      -- (C24);
\draw(1)      -- (C24);
\draw(1)      -- (C23);
\draw(1)      -- (C22);
\draw(1)      -- (C21);
\draw(1)      -- (C2);
\end{tikzpicture}
    \caption{Subgroup Lattice of $D_4$}
    \label{fig:lattice-d4}
\end{figure}

\begin{figure}[h]
    \centering
    \begin{tikzpicture}[node distance=2cm,line width=0.5pt]
\node(D6) at (0,0)     {$D_6$};
\node(O61)       [below right =1cm and 2cm of D4] {$\langle r^2, sr\rangle$};
\node(O62)      [below left =1cm and 2cm of D4]  {$\langle r^2, s\rangle$};
\node(O63)      [below = 1cm of D6]       {$\langle r\rangle$};
\node(O41)      [below left =1cm and 0.5cm of O63]       {$\langle r^3, s\rangle$};
\node(O42)      [below right =1cm and 0.5cm of O63]       {$\langle r^3, sr^2\rangle$};
\node(O43)      [below right =1cm and 3.5cm of O61]       {$\langle r^3, sr\rangle$};
\node(O3)      [below left =2cm and 0.25cm of O63]       {$\langle r^2\rangle$};
\node(O21)      [below =3cm of O62]       {$\langle  sr^2\rangle$};
\node(O22)      [left =1cm of O21]       {$\langle  sr^4\rangle$};
\node(O23)      [left =1cm of O22]       {$\langle  s\rangle$};
\node(O2)      [right =2cm of O21]       {$\langle  r^3\rangle$};
\node(O24)      [right =2cm of O2]       {$\langle  sr^5\rangle$};
\node(O25)      [right =1cm of O24]       {$\langle  sr^3\rangle$};
\node(O26)      [right =1cm of O25]       {$\langle  sr\rangle$};
\node(1)        [below=2cm of O3]   {$\left\{1\right\}$};
\draw(D6)       -- (O61);
\draw(D6)       -- (O62);
\draw(D6)       -- (O63);
\draw(D6)       -- (O41);
\draw(D6)       -- (O42);
\draw(D6)       -- (O43);
\draw(O61)       -- (O24);
\draw(O61)       -- (O25);
\draw(O61)       -- (O26);
\draw(O61)       -- (O3);
\draw(O62)       -- (O21);
\draw(O62)       -- (O22);
\draw(O62)       -- (O23);
\draw(O62)       -- (O3);
\draw(O63)       -- (O3);
\draw(O63)       -- (O2);
\draw(O41)       -- (O2);
\draw(O41)       -- (O25);
\draw(O41)       -- (O23);
\draw(O42)       -- (O2);
\draw(O42)       -- (O21);
\draw(O42)       -- (O24);
\draw(O43)       -- (O2);
\draw(O43)       -- (O22);
\draw(O43)       -- (O26);
\draw(1)      -- (O2);
\draw(1)      -- (O21);
\draw(1)      -- (O22);
\draw(1)      -- (O23);
\draw(1)      -- (O24);
\draw(1)      -- (O25);
\draw(1)      -- (O26);
\draw(1)      -- (O3);
\end{tikzpicture}
    \caption{Subgroup Lattice of $D_6$}
    \label{fig:lattice-d6}
\end{figure}

\pagebreak

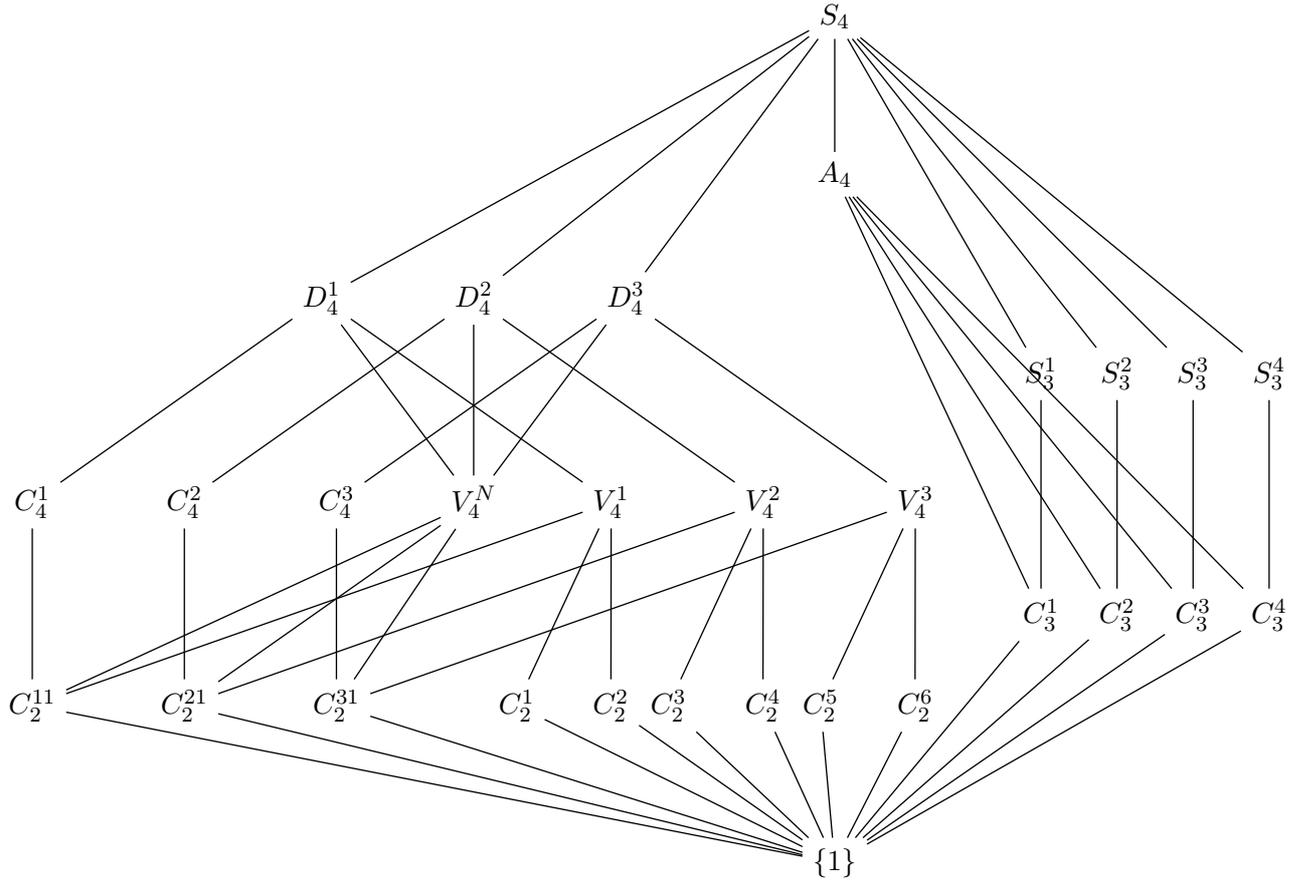
\begin{figure}[h]
    \centering
    \begin{tikzpicture}[node distance=2cm,line width=0.5pt]
\node(S4) at (0,0)     {$S_4$};
\node(A4)       [below =1.5cm of S4] {$A_4$};
\node(D41)  [below left = 1cm and 6cm of A4] {$D_4^1$};
\node(D42)  [below left = 1cm and 4cm of A4] {$D_4^2$};
\node(D43)  [below left = 1cm and 2cm of A4] {$D_4^3$};
\node(S31)  [below right = 2cm and 2cm of A4] {$S_3^1$};
\node(S32)  [below right = 2cm and 3cm of A4] {$S_3^2$};
\node(S33)  [below right = 2cm and 4cm of A4] {$S_3^3$};
\node(S34)  [below right = 2cm and 5cm of A4] {$S_3^4$};
\node(V4N)  [below = 2cm of D42] {$V_4^N$};
\node(C43)  [left = 1cm of V4N] {$C_4^3$};
\node(C42)  [left = 3cm of V4N] {$C_4^2$};
\node(C41)  [left = 5cm of V4N] {$C_4^1$};
\node(V43)  [right = 5cm of V4N] {$V_4^3$};
\node(V42)  [right = 3cm of V4N] {$V_4^2$};
\node(V41)  [right = 1cm of V4N] {$V_4^1$};
\node(C211) [below = 2cm of C41] {$C_2^{11}$};
\node(C221) [below = 2cm of C42] {$C_2^{21}$};
\node(C231) [below = 2cm of C43] {$C_2^{31}$};
\node(C21)  [below left = 2cm and 0.5cm of V41] {$C_2^1$};
\node(C22)  [right = 0.5cm of C21] {$C_2^2$};
\node(C23)  [below left = 2cm and 0.5cm of V42] {$C_2^3$};
\node(C24)  [right = 0.5cm of C23] {$C_2^4$};
\node(C25)  [below left = 2cm and 0.5cm of V43] {$C_2^5$};
\node(C26)  [right = 0.5cm of C25] {$C_2^6$};
\node(C31)  [below = 2.5cm of S31] {$C_3^1$};
\node(C32)  [below = 2.5cm of S32] {$C_3^2$};
\node(C33)  [below = 2.5cm of S33] {$C_3^3$};
\node(C34)  [below = 2.5cm of S34] {$C_3^4$};
\node(C1)   [below = 8.5cm of A4] {$\{1\}$};
\draw(S4)       -- (A4);
\draw(S4)       -- (D41);
\draw(S4)       -- (D42);
\draw(S4)       -- (D43);
\draw(S4)       -- (S31);
\draw(S4)       -- (S32);
\draw(S4)       -- (S33);
\draw(S4)       -- (S34);
\draw(D41)       -- (V4N);
\draw(D42)       -- (V4N);
\draw(D43)       -- (V4N);
\draw(D41)       -- (C41);
\draw(D42)       -- (C42);
\draw(D43)       -- (C43);
\draw(D41)       -- (V41);
\draw(D42)       -- (V42);
\draw(D43)       -- (V43);
\draw(C41)       -- (C211);
\draw(V4N)       -- (C211);
\draw(V41)       -- (C211);
\draw(C42)       -- (C221);
\draw(V4N)       -- (C221);
\draw(V42)       -- (C221);
\draw(C43)       -- (C231);
\draw(V4N)       -- (C231);
\draw(V43)       -- (C231);
\draw(V41)       -- (C21);
\draw(V41)       -- (C22);
\draw(V42)       -- (C23);
\draw(V42)       -- (C24);
\draw(V43)       -- (C25);
\draw(V43)       -- (C26);
\draw(S31)       -- (C31);
\draw(A4)       -- (C31);
\draw(S32)       -- (C32);
\draw(A4)       -- (C32);
\draw(S33)       -- (C33);
\draw(A4)       -- (C33);
\draw(S34)       -- (C34);
\draw(A4)       -- (C34);
\draw(C1)       -- (C211);
\draw(C1)       -- (C221);
\draw(C1)       -- (C231);
\draw(C1)       -- (C21);
\draw(C1)       -- (C22);
\draw(C1)       -- (C23);
\draw(C1)       -- (C24);
\draw(C1)       -- (C25);
\draw(C1)       -- (C26);
\draw(C1)       -- (C31);
\draw(C1)       -- (C32);
\draw(C1)       -- (C33);
\draw(C1)       -- (C34);
\end{tikzpicture}
    \caption{Subgroup Lattice of $S_4$ (subgroups of $S_3$ omitted)}
    \label{fig:lattice-s4}
\end{figure}

\end{document}